\documentclass[12pt,a4paper]{amsart}
\usepackage{mathrsfs}
\usepackage{amssymb,amsmath,amsthm,color}
\usepackage{graphicx,mcite}
\usepackage{hyperref}
\usepackage{url}
\usepackage{setspace}
\usepackage{enumerate}
\newtheorem{theorem}{Theorem}[section]
\newtheorem{lemma}[theorem]{Lemma}
\newtheorem{proof of lemma}[theorem]{Proof of Lemma}
\newtheorem{proposition}[theorem]{Proposition}

\newtheorem{corollary}[theorem]{Corollary}

\theoremstyle{definition}
\newtheorem{definition}[theorem]{Definition}

\newtheorem{remark}[theorem]{Remark}

\numberwithin{equation}{section}

\makeatletter
\@namedef{subjclassname@2020}{\textup{2020} Mathematics Subject Classification}
\makeatother

\begin{document}

\title[Weyl transform]
{Unbounded Weyl transform on the Euclidean motion group and Heisenberg motion group}

\author{Somnath Ghosh and R.K. Srivastava}

\address{Department of Mathematics, Indian Institute of Technology, Guwahati, India 781039.}
\email{somnath.g.math@gmail.com, rksri@iitg.ac.in}


\date{\today}


\begin{abstract}
In this article, we define Weyl transform on second countable type - $I$ locally compact group $G,$
and as an operator on $L^2(G),$ we prove that the Weyl transform is compact when the symbol
lies in $L^p(G\times \hat{G})$ with $1\leq p\leq 2.$ Further, for the Euclidean motion group and
Heisenberg motion group, we prove that the Weyl transform can not be extended as a bounded
operator for the symbol belongs to $L^p(G\times \hat{G})$ with $2<p<\infty.$ To carry out this,
we construct positive, square integrable and compactly supported function, on the respective groups,
such that $L^{p'}$ norm of its Fourier transform is infinite, where $p'$ is the conjugate index of $p.$
\end{abstract}

\maketitle

\section{Introduction}
In \cite{We}, Hermann Weyl studied the quantization problem in quantum mechanics and introduced
a type of pseudo-differential operators. These operators are useful in physics and mathematics,
especially in PDE, harmonic analysis, time frequency analysis. In \cite{Wo}, Wong called these
operators as Weyl transform. Further in \cite{Wo}, for the symbol in $L^p(\mathbb{R}^{2n})$ with
$1\leq p\leq 2,$ compactness of the Weyl transform, as an operator on $L^2(\mathbb{R}^n),$ is
studied. Moreover in \cite{Si}, Simon proved that the operator is not even bounded when the symbol
is in $L^p(\mathbb{R}^{2n})$ with $2<p<\infty.$

\smallskip

In general, for a locally compact group, Fourier transform is an operator valued function. For the
Heisenberg group, in \cite{PZ}, Weyl transform is defined for operator valued symbol and proved its
boundedness (even compactness) when the symbol is in the corresponding $L^p$ spaces with $1\leq p\leq 2,$
while unboundedness for $2<p<\infty.$ Further, in \cite{PZ2} Weyl transform on the upper half plane,
and in \cite{CZ} Weyl transform on the quaternion Heisenberg group, are studied.

\smallskip

In this article, we consider second countable type $I$ locally compact group $G$ and define Weyl
transform in view of its inversion formula. We prove that the Weyl transform, as an operator on
$L^2(G),$ is compact, when the symbol is in $L^p(G\times \hat{G})$ with $1\leq p\leq 2,$ where
$\hat{G}$ is the dual of $G.$ Further, to prove that the Weyl transform can not be extended as
a bounded operator for the symbol is in $L^p(G\times \hat{G})$ with $2<p<\infty,$ it is enough to
construct a positive, square integrable and compactly supported function on $G$ such that $L^{p'}$
norm of its Fourier transform is infinite, where $\frac{1}{p}+\frac{1}{p'}=1.$ In this perspective,
we construct such type of functions for the Euclidean motion group and Heisenberg motion group.

\smallskip

In addition, we want to mention that these examples will not follow directly from the Euclidean space
and Heisenberg group, respectively, for the following reasons. The Fourier transform on the Euclidean
motion group is operator valued, whereas it is just a function for the Euclidean space. Secondly, the 
case becomes more difficult for the Heisenberg motion group due to the presence of metaplectic 
representation whose implicitness may not be visible at the first instance.

\section{Weyl transform on locally compact groups}
In this section, we recall some harmonic analysis results, namely Fourier inversion, Plancherel formula,
and Hausdorff-Young inequality on certain locally compact groups. Then in terms of Wigner transform, we
define Weyl transform. After that, we prove the compactness of the Weyl transform for the symbol in
corresponding $L^p$ spaces with $1\leq p\leq 2.$ This section concludes with a sufficient condition for
the unboundedness of the Weyl transform for $2<p<\infty.$

\smallskip

Let $G$ be a second countable locally compact group with type $I$ left regular representation, and
$A(G),$ $\hat{G}$ denote the Fourier algebra and dual of $G,$ respectively. Then there is a standard
measure $\mu$ on $\hat{G},$ called the Plancherel measure, a $\mu$-measurable field $(\pi,\mathcal{H}_\pi)$
of representation and a measurable field $\mathcal{K}=(K_\pi)$ of nonzero positive self-adjoint operators
such that $K_\pi$ is semi-invariant with weight $\triangle^{-1},$ for almost all $\pi \in \hat{G},$ where
$\triangle$ be the modular function of $G$ and $A(G).$ If $G$ is unimodular, then $K_\pi=I_{\mathcal{H}_{\pi}},$
the identity operator on $\mathcal{H}_{\pi}.$ For $f\in L^1(G) \cap A(G),$ the group Fourier transform
\[\pi(f)=\int_G f(x)\pi(x) d\nu(x)\] is a bounded operator on $\mathcal{H}_\pi,$ where $d\nu$ is a left Haar
measure on $G.$ Further, $f$ can be recovered by the inversion formula.
\begin{theorem}\cite{N}(Inversion theorem)\label{th151}
Let $f\in L^1(G) \cap A(G).$ Then
\[f(x)=\int_{\hat{G}} \text{tr} (\pi(x)^{-1} \pi(f)K_{\pi})d\mu(\pi).\]
\end{theorem}
Next, we discuss the Plancherel formula for the Fourier transform, and for this, we start with Schatten
class operators. Let $S_p$ be the space of all $p$-Schatten class operators, which is a Banach space
with the norm $\|T\|_{S_p}^p=\text{tr}(T^*T)^{p/2}$ for $1\leq p <\infty$ and for $p=\infty,$ the norm
is the usual operator norm.
\begin{theorem}\cite{DM}(Plancherel formula)\label{th152}
Let $f\in L^1\cap L^2(G).$ Then
\[ \int_G |f(x)|^2 d\nu(x)=\int_{\hat{G}}\|\pi(f)K_\pi^{1/2}\|_{S_2}^2 ~d\mu(\pi).\]
\end{theorem}
This map $L^1\cap L^2(G)\rightarrow L^2(\hat{G})$ can be extended uniquely to a unitary map
from $L^2(G)$ onto $L^2(\hat{G}).$

\smallskip

Let $f\in L^1\cap L^p(G),$ where $1\leq p\leq \infty.$ Define the $L^p$ Fourier transform of $f$ by
$\mathcal{F}_p(f)\pi=\pi(f)K_\pi^{1/p'},$ where $\frac{1}{p}+\frac{1}{p'}=1.$ Then the following
Hausdorff-Young inequality holds.
\begin{theorem}\cite{N}\label{th153}
If $f\in L^1\cap L^p(G),$ where $1<p<2,$ then $\mathcal F_p(f) \in L^{p'}(\hat{G})$ and this map
$f\mapsto \mathcal F_p(f)$ extends uniquely to a bounded linear map from $L^p(G)$ into
$L^{p'}(\hat{G})$ with norm less than or equal to $1.$
\end{theorem}
To study the Weyl transform, we need to define Wigner transform on $G$ and investigate its boundedness
properties. For this, we first describe the following product spaces and the left translation operator.

\smallskip

Consider the measure $d\nu \otimes d\mu$ on $G \times \hat{G}$ and for $1\leq p\leq \infty,$ let
$L^p(G \times \hat{G},S_p,d\nu \otimes d\mu)$ be the space of $S_p$ valued functions satisfying
\begin{align*}
\|f\|_{p,\nu \otimes \mu}^p=&\int_{G \times \hat{G}}\|f(x,\pi)K_\pi^{1/p}\|_{S_p}^p d\nu(x)d\mu(\pi)<\infty,
\qquad 1\leq p<\infty, \\
\|f\|_{\infty,\nu \otimes \mu}=&\underset{(x,\pi)\in G\times\hat{G}}{\text{ess sup}}\|f(x,\pi)\|_{S_\infty}<\infty.
\end{align*}
The left translation operator is defined by $\tau_{x'}f(x)=f(x'^{-1}x)$ for $f\in L^p(G).$ Further,
$C_c(G)$ denotes the space of all compactly supported continuous functions on $G.$ Throughout this
section, $p,p'$ are the conjugate indices.

\begin{definition}
Let $f,g\in C_c(G)$ and $(x,\pi)\in G\times\hat{G}.$ Then Wigner transform associated with $f,g$ is defined by
\[V(f,g)(x,\pi)=\int_G f(x')\tau_{x'}g(x)\pi(x')d\nu(x').\]
That is, $V(f,g)(x,\pi)=\pi{\left(f\cdot \tau g(x)\right)}.$
\end{definition}
The following result proves that the Wigner transforms are in certain $L^p$ spaces.
\begin{proposition}\label{prop151}
Let $f,g \in C_c(G).$ Then $V(f,g)\in L^{p'}(G \times \hat{G},S_{p'},d\nu \otimes d\mu)$ and
\begin{align}\label{exp151}
\|V(f,g)\|_{p',\nu \otimes \mu} \leq \|f\|_2 \|g\|_2
\end{align}
for $p'\in [2,\infty].$ Thus $V:C_c(G)\times C_c(G) \rightarrow L^{p'}(G \times \hat{G},S_{p'},d\nu \otimes d\mu)$
can be extended uniquely to a bilinear operator
$V:L^2(G) \times L^2(G) \rightarrow L^{p'}(G \times \hat{G},S_{p'},d\nu \otimes d\mu)$ with
\begin{align*}
\|V(f,g)\|_{p',\nu \otimes \mu} \leq \|f\|_2\|g\|_2.
\end{align*}
\end{proposition}
\begin{proof}
Let $p'=\infty.$ Then
\begin{align*}
\|V(f,g)\|_{\infty,\nu \otimes \mu}=& \underset{(x,\pi)\in G \times \hat{G}}{\text{ess sup}}\|V(f,g)(x,\pi)\|_{S_\infty} \\
=& \underset{(x,\pi)\in G \times \hat{G}}{\text{ess sup}}\|\int_G f(x')\tau_{x'}g(x)\pi(x')d\nu(x')\|_{S_\infty} \\
\leq & \, \underset{x\in G}{\text{ess sup}}\int_G |f(x')\tau_{x'}g(x)| d\nu(x') \leq \|f\|_2 \|g\|_2.
\end{align*}
For $p'=2,$ Plancherel formula, Theorem \ref{th152}, gives
\begin{align*}
\|V(f,g)\|_{2,\nu \otimes \mu}^2
=&\int_G \int_{\hat{G}}\|\pi{\left(f\cdot \tau g(x)\right)}K_{\pi}^{1/2}\|_{S_2}^2 d\mu(\pi) d\nu(x) \\
=&\int_G \int_G |f(x')\tau_{x'}g(x)|^2 d\nu(x')d\nu(x)
=\|f\|_2^2\|g\|_2^2.
\end{align*}
Then Riesz-Thorin interpolation theorem completes the proof.
\end{proof}

The following proposition is a way of writing the group Fourier transform in terms of the Wigner transform.
\begin{proposition}\label{prop152}
Let $f,g\in L^1\cap L^2(G)$ and $\mathrm C=\int_G g(x) d\nu(x)\neq 0.$
Then $\pi(f)=\mathrm C^{-1}\int_G V(f,g)(x,\pi)d\nu(x)$ for $\pi\in \hat{G}.$
\end{proposition}
\begin{proof}
\begin{align*}
\int_G V(f,g)(x,\pi)d\nu(x)=&\int_G \int_G f(x')g(x'^{-1}x)\pi(x') d\nu(x)d\nu(x') \\
=&\left(\int_G g(x) d\nu(x)\right) \left(\int_G f(x')\pi(x')d\nu(x')\right) =\mathrm C~\pi(f).
\end{align*}
\end{proof}
In view of Proposition \ref{prop152}, the Inversion formula \ref{th151} can be reformulated.
\begin{corollary}
Let $f\in L^1(G)\cap A(G)$ and $g\in L^1\cap L^2(G)$ with $\mathrm C=\int_G g(x) d\nu(x)\neq 0.$ Then
\[f(x)=\mathrm C^{-1}\int_{\hat{G}}\text{tr}\left(\pi(x)^{-1}\left(\int_G V(f,g)(x',\pi)d\nu(x')\right)
K_{\pi} \right)d\mu(\pi).\]
\end{corollary}

\begin{definition}\label{def151}
Let $\varsigma\in L^p(G \times \hat{G},S_p,d\nu \otimes d\mu),$ where $1\leq p\leq 2.$ Then corresponding
to $\varsigma,$ Weyl transform $W_{\varsigma}:L^2(G)\rightarrow L^2(G)$ is defined by
\begin{align*}
\langle W_{\varsigma}f,\bar{g}\rangle =\langle V(f,g),\varsigma \rangle_{\nu \otimes \mu}
=\int_G \int_{\hat{G}}\text{tr}\left(\varsigma^*(x,\pi)V(f,g)(x,\pi)K_{\pi}\right)d\mu(\pi)d\nu(x),
\end{align*}
where $f,g\in L^2(G).$ Here, by the abuse of notation $\langle \cdot\, , \cdot  \rangle_{\nu \otimes \mu}$ is used.
\end{definition}
After a bit of calculation, we can conclude that
\begin{align}\label{exp152}
W_{\varsigma}f(x)=\int_{\hat{G}}\int_G \text{tr}\left(\varsigma^*(x'x,\pi)
\pi(x')K_{\pi}\right)f(x')d\nu(x')d\mu(\pi).
\end{align}
Thus $W_{\varsigma}:L^2(G)\rightarrow L^2(G)$ is an integral operator with kernel
\[K(x,x')=\int_{\hat{G}}\text{tr}\left(\varsigma^*(x'x,\pi)\pi(x')K_{\pi}\right)d\mu(\pi).\]
In the following proposition, we investigate the boundedness of the Weyl transform.
\begin{proposition}\label{prop153}
Let $\varsigma\in L^p(G \times \hat{G},S_p,d\nu \otimes d\mu),$ where $1\leq p \leq 2.$
Then $\|W_{\varsigma}\|\leq \|\varsigma\|_{p,\nu \otimes \mu}.$
\end{proposition}
\begin{proof}
Since $|\text{tr}(A^*B)|\leq \|A\|_{S_p}\|B\|_{S_{p'}},$ for $f,g \in L^2(G)$ we have
\begin{align*}
|\langle W_{\varsigma}f,\bar{g}\rangle| &\leq \int_G \int_{\hat{G}}|\text{tr}
\left(\varsigma^*(x,\pi)V(f,g)(x,\pi)K_{\pi}\right)|d\mu(\pi)d\nu(x) \\
&\leq \int_G \int_{\hat{G}}\|\varsigma(x,\pi)(K_{\pi}^{1/p})^*\|_{S_p}\|V(f,g)(x,\pi)K_{\pi}^{1/p'}\|_{S_{p'}}d\mu(\pi)d\nu(x).
\end{align*}
By H\"{o}lder's inequality, the above integral is less than or equal to
\begin{align*}
&\left(\int_G \int_{\hat{G}} \|\varsigma(x,\pi)K_{\pi}^{1/p}\|_{S_p}^p d\mu(\pi)d\nu(x)\right)^{\frac{1}{p}}
\left(\int_G \int_{\hat{G}}\|V(f,g)(x,\pi)K_{\pi}^{1/p'}\|_{S_{p'}}^{p'} d\mu(\pi)d\nu(x)\right)^{\frac{1}{p'}} \\
&=\|\varsigma \|_{p,\nu \otimes \mu} \|V(f,g)\|_{p',\nu \otimes \mu}
\leq \|\varsigma \|_{p,\nu \otimes \mu} \|f\|_2 \|g\|_2.
\end{align*}
\end{proof}
Further for $p=2,$ $W_\varsigma$ is a Hilbert-Schmidt operator, and for $p=1,$ $W_\varsigma$ is a trace class operator.
\begin{proposition}\label{prop156}
If $\varsigma\in L^2(G \times \hat{G},S_2,d\nu \otimes d\mu),$ then $W_{\varsigma}$
is a Hilbert-Schmidt operator.
\end{proposition}
\begin{proof}
Since $W_\varsigma$ is an integral operator, we have
\begin{align*}
\|W_{\varsigma}\|_{S_2}^2 &=\int_G\int_G|K(x,x')|^2 d\nu(x) d\nu(x') \\
&=\int_G\int_G \left|\int_{\hat{G}}\text{tr}\left(\varsigma^*(x'x,\pi)\pi(x')K_{\pi}\right)d\mu(\pi)\right|^2 d\nu(x) d\nu(x') \\
&=\int_G\int_G \left|\int_{\hat{G}}\text{tr}\left(\varsigma^*(x,\pi)\pi(x')K_{\pi}\right)d\mu(\pi)\right|^2 d\nu(x) d\nu(x'),
\end{align*}
where last equality is ensured by applying the change of variable $x=x'x,$ as $d\nu$ is a left invariant
Haar measure. Since $\text{tr}(T^*)=\overline{\text{tr}(T)}$ and $K_\pi^*=K_\pi,$ the above integral becomes
\begin{align*}
\int_G\int_G \left|\int_{\hat{G}}\text{tr}\left(K_{\pi}\pi(x')^{-1}\varsigma(x,\pi)\right)d\mu(\pi)\right|^2 d\nu(x)d\nu(x').
\end{align*}
Applying the Fourier inversion and Plancherel formula, we get
\begin{align*}
\|W_{\varsigma}\|_{S_2}^2&=\int_G\int_G |\mathcal{F}^{-1}(\varsigma)(x,\cdot)(x')|^2 d\nu(x') d\nu(x) \\
&=\int_G\int_{\hat{G}}\|\varsigma(x,\pi)K_{\pi}^{1/2}\|_{S_2}^2 d\mu(\pi)d\nu(x)=\|\varsigma\|_{2,\nu\otimes \mu}^2 .
\end{align*}
\end{proof}

\begin{proposition}
If $\varsigma\in L^1(G \times \hat{G},S_1,d\nu \otimes d\mu),$ then $W_{\varsigma}$
is a trace class operator.
\end{proposition}
\begin{proof}
Since $W_\varsigma$ is an integral operator
\begin{align*}
\|W_{\varsigma}\|_{S_1} &=\int_G |K(x,x)| d\nu(x) \\
&=\int_G \left|\int_{\hat{G}}\text{tr}\left(\varsigma^*(xx,\pi)\pi(x)K_{\pi}\right)d\mu(\pi)\right| d\nu(x) \\
&\leq \int_G \int_{\hat{G}}\left|\text{tr}\left(\pi(x)K_{\pi}\varsigma^*(xx,\pi)\right)\right|d\mu(\pi) d\nu(x) \\
&\leq \int_G \int_{\hat{G}}\text{tr}\left(\left|\varsigma^*(xx,\pi)K_{\pi}\right|\right)d\mu(\pi) d\nu(x)
=\|\varsigma\|_{1,\nu \otimes \mu},
\end{align*}
where the second last inequality true because of $|\text{tr}(T)|\leq \text{tr}(T^*T)^{1/2}=\text{tr}(|T|).$
\end{proof}

\begin{theorem}
If $\varsigma\in L^p(G \times \hat{G},S_p,d\nu \otimes d\mu),$ where $1\leq p \leq 2,$
then $W_{\varsigma}$ is a compact operator.
\end{theorem}
\begin{proof}
Let $\varsigma\in L^p(G \times \hat{G},S_p,d\nu \otimes d\mu)$ and $(\varsigma_n)$ be a sequence
in $C_c(G \times \hat{G})$ such that $\|\varsigma_n - \varsigma\|_{p,\nu \otimes \mu}\rightarrow 0.$
From Proposition \ref{prop153}, it follows that $W_{\varsigma_n}$ converges to $W_\varsigma$ in
operator norm. But in view of Proposition \ref{prop156}, $W_{\varsigma_n}$ is a compact operator for each $n.$
Hence $W_{\varsigma}$ is a compact operator.
\end{proof}

Now, we shall see whether the Weyl transform defined in (\ref{exp152}) can be extended as a bounded operator
for $p>2.$ Next result gives a necessary and sufficient condition for boundedness of the Weyl transform in
terms of the Wigner transform.
\begin{proposition}\label{prop154}
Let $2< p<\infty$ and $\frac{1}{p}+\frac{1}{p'}=1.$ Then for all
$\varsigma\in L^p(G \times \hat{G},S_p,d\nu \otimes d\mu),$ the Weyl
transform $W_{\varsigma}$ is a bounded operator on $L^2(G)$ if and
only if there exists a constant $C$ such that
$\|V(f,g)\|_{p',\nu \otimes \mu}\leq C\|f\|_2\|g\|_2$ for all $f,g\in L^2(G).$
\end{proposition}

\begin{proof}
A similar argument as in Proposition \ref{prop153} proves that $W_{\varsigma}$ is bounded for all
$\varsigma\in L^p(G \times \hat{G},S_p,d\nu \otimes d\mu),$ whenever there exists a constant $C$
such that $\|V(f,g)\|_{p',\nu \otimes \mu}\leq C\|f\|_2\|g\|_2$ for all $f,g\in L^2(G).$

Conversely, assume that  $W_{\varsigma}$ is bounded for each $\varsigma\in L^p(G \times \hat{G},S_p,d\nu \otimes d\mu).$
Then there exists a constant $C_\varsigma$ such that $\|W_{\varsigma}f\|_2\leq C_\varsigma \|f\|_2$ for all $f\in L^2(G).$
For $f,g\in C_c(G)$ with $\|f\|_2=\|g\|_2=1,$ define the bounded linear functional on
$L^p(G \times \hat{G},S_p,d\nu \otimes d\mu)$ by $\mathcal{Q}_{f,g}(\varsigma)=\langle W_\varsigma f,g\rangle.$
Then $\sup |\mathcal{Q}_{f,g}(\varsigma)|\leq C_\varsigma,$ where the supremum is over all $f,g\in C_c(G)$ with
$\|f\|_2=\|g\|_2=1.$ Therefore, by the uniform boundedness principle, there exists a constant $C$ such that
$\|\mathcal{Q}_{f,g}\|\leq C$ for all $f,g\in C_c(G)$ with $\|f\|_2=\|g\|_2=1.$ Hence
$|\langle W_\varsigma f,g\rangle|\leq C\|\varsigma\|_{p,\nu \otimes \mu}\|f\|_2\|g\|_2,$
that is, $\|W_\varsigma\|\leq C\|\varsigma\|_{p,\nu \otimes \mu}.$ Thus for $f,g\in C_c(G),$
\begin{align*}
\|V(f,g)\|_{p',\nu \otimes \mu}=\sup\limits_{\|\varsigma\|_{p,\nu \otimes \mu}=1}
|\langle V(f,g),\varsigma \rangle_{\nu \otimes \mu}|
=\sup\limits_{\|\varsigma\|_{p,\nu \otimes \mu}=1}|\langle W_\varsigma f,\bar{g}\rangle| \leq C\|f\|_2\|g\|_2.
\end{align*}
The density argument completes the proof.
\end{proof}
A sufficient condition for unboundedness of the Weyl transform is obtained in the following result.
\begin{proposition}\label{prop155}
Let $2<p<\infty$ and $f$ be a square integrable, compactly supported function on $G$ with
$\int_G f(x)d\nu(x)\neq 0.$ If $W_\varsigma$ is a bounded operator on $L^2(G)$ for all
$\varsigma\in L^p(G \times \hat{G},S_p,d\nu \otimes d\mu),$
then $\int_{\hat{G}}\|\mathcal{F}_p(f)\pi\|_{S_{p'}}^{p'} ~d\mu(\pi)<\infty.$
\end{proposition}

\begin{proof}
Let $f$ be supported on a compact set $K.$ Then $\tilde{K}=KK=\{xy:x,y\in K\}$ is compact and $V(f,f)$
is supported on $\tilde{K}\times \hat{G}.$ In view of Proposition \ref{prop152}, instead of
$\int_{\hat{G}}\|\mathcal{F}_p(f)\pi\|_{S_{p'}}^{p'} ~d\mu(\pi),$ it is enough to consider the
following integral
\begin{align*}
\int_{\hat{G}}\left\|\int_G V(f,f)(x,\pi)K_\pi^{1/p'}d\nu(x)\right\|_{S_{p'}}^{p'}d\mu(\pi).
\end{align*}
Applying Minkowski's integral inequality and H\"{o}lder's inequality, the above integral is less than or equal to
\begin{align*}
&\left(\int_{\tilde{K}}\left(\int_{\hat{G}}\|V(f,f)(x,\pi)K_\pi^{1/p'}\|_{S_{p'}}^{p'}d\mu(\pi)\right)^{\frac{1}{p'}}
d\nu(x)\right)^{p'} \\
\leq &\left(\int_{\tilde{K}}d\nu(x)\right)^{p'/p}\int_{\tilde{K}}\int_{\hat{G}}\|V(f,f)(x,\pi)K_\pi^{1/p'}\|_{S_{p'}}^{p'}
d\mu(\pi)d\nu(x).
\end{align*}
Hence by Proposition \ref{prop154}, it completes the proof.
\end{proof}

To prove that the Weyl transform $W_\varsigma$ can not be extended as a bounded operator for
$\varsigma\in L^p(G \times \hat{G},S_p,d\nu \otimes d\mu),$ where $2<p<\infty,$ it is enough
to consider the following problem.

\smallskip

\noindent\textbf{Question :} For $p\in(2,\infty),$ does there exists a square integrable, compactly supported function
$f$ on $G$ with $\int_G f(x)d\nu(x)\neq 0,$ such that $\int_{\hat{G}}\|\mathcal{F}_p(f)\pi\|_{S_{p'}}^{p'}~d\mu(\pi)=\infty$ ?

\smallskip

In \cite{Si}, Simon gave the example of such functions for $\mathbb{R}^n,$ and later on, it is considered
for the Heisenberg group \cite{PZ}, quaternion Heisenberg group \cite{CZ} and upper half plane \cite{PZ2}.
In this article, we define such functions for the Euclidean motion group and Heisenberg motion group.

\section{Euclidean motion group}
In this section, we briefly discuss the Fourier analysis on the Euclidean motion group. Thereafter, we
shall precisely write down the formula for the Weyl transform in this setup and prove that it can not
be extended as a bounded operator for the symbol in the corresponding $L^p$ spaces with $2<p<\infty.$

\smallskip

Let $n\in\mathbb N\setminus \{1\}$ and $SO(n)$ be the special orthogonal group of order $n.$ Then Euclidean
motion group $M(n)$ is the semidirect product of $\mathbb{R}^n$ with $K=SO(n).$ The group law on $M(n)$ can be
expressed as
\[(x_1,k_1)(x_2,k_2)=(x_1+k_1\cdot x_2,k_1k_2),\]
where $x_1,x_2\in\mathbb R^n$ and $k_1,k_2\in K.$ The Haar measure on $M(n)$ can be written as $d\nu(x,k)=dxdk,$
where $dx$ and $dk$ are the Haar measures on $\mathbb R^n$ and $K$ respectively. The Fourier analysis on
$M(n)$ can be discussed in the following two cases. For details, see \cite{KK,ST,Su}.

\smallskip

\noindent\textbf{For $n=2$ :} An arbitrary element of $M(2)$ can be written as $(z,e^{i\varphi}),$ where
$z\in \mathbb C,\varphi\in \mathbb R.$ All the, upto unitarily equivalent, infinite dimensional irreducible
unitary representations of $M(2)$ are parametrized by $a>0.$ Moreover, for each $a>0,$ the representation
$\pi_a,$ realized on $L^2(S^1),$ is given by
\[\pi_a(z,e^{i\varphi})g(\theta)=e^{i\text{Re}(ae^{i\theta}z)}g(\theta-\varphi),\]
where $g\in L^2(S^1).$ The Plancherel measure on $(0,\infty)$ is $d\mu(a)=ada,$
where $da$ is the Lebesgue measures on $(0,\infty).$ Further, for $f\in L^1(M(2))$ the group
Fourier transform, defined by
\begin{align*}
\pi_a(f)=\hat{f}(a)=\int_{M(2)}f(z,e^{i\varphi})\pi_a(z,e^{i\varphi})dzd\varphi,
\end{align*}
is a bounded operator on $L^2(S^1).$ For $\varsigma\in L^p(M(2) \times (0,\infty),S_p,d\nu \otimes d\mu)$
and $f\in L^2(M(2)),$ the Weyl transform $W_\varsigma$ takes the form
\begin{align*}
W_\varsigma f(z,e^{i\varphi})=\int_0^\infty \int_{M(2)}\text{tr}\left(\varsigma^*((w,e^{i\psi})(z,e^{i\varphi}))
\pi_a(w,e^{i\psi})\right)f(w,e^{i\psi})a dw d\psi da.
\end{align*}

\smallskip

\noindent\textbf{For $n\geq 3$ :} Let $M=SO(n-1)$ be the subgroup of $K$ that fixes the
point $e_1=(1,0,\ldots,0).$ Let $\hat M$ be the unitary dual group of $M.$ Given a unitary
irreducible representation $\sigma\in\hat M,$ realized on the Hilbert space $\mathcal H_\sigma$
of dimension $d_\sigma,$ consider the space $L^2(K,\mathbb C^{d_\sigma\times d_\sigma})$
consisting of $d_\sigma\times d_\sigma$ complex matrices valued functions $\varphi$ on $K$
such that $\varphi(uk)=\sigma(u)\varphi(k),$ where $u\in M,~k\in K,$ and satisfying
\[\int_K\text{tr}(\varphi(k)^\ast\varphi(k))dk<\infty.\]
For $(a,\sigma)\in (0,\infty)\times\hat M,$ a unitary representation $\pi_{a,\sigma}$
of $M(n)$ is defined by
\begin{equation*}
\pi_{a,\sigma}(x,k)(\varphi)(s)=e^{ia\left\langle s^{-1}\cdot e_1,\, x\right\rangle}\varphi(sk),
\end{equation*}
where $\varphi\in L^2\left(K,\mathbb C^{d_\sigma\times d_\sigma}\right).$ These are all, up to
unitary equivalent, infinite dimensional, unitary, irreducible representations that appear in
the Plancherel formula. The group Fourier transform of a function $f\in L^1(M(n))$ is defined by
\[\hat f(a,\sigma)=\int_{M(n)} f(x,k)\pi_{a,\,\sigma}(x,k)dx dk.\]
Further, for $f\in L^1\cap L^2(M(n)),$ Plancheral formula holds
\[\int_{M(n)} |f(x,k)|^2 dx dk=\mathsf{c}_n\int_0^\infty
\left(\sum_{\sigma\in\hat{M}}d_\sigma\|\hat f(a,\sigma)\|_{HS}^2 \right)a^{n-1}da,\]
for some constant $\mathsf{c}_n.$

\smallskip

Next, we construct the required square integrable and compactly supported functions for $M(n).$ To
do so, we need to discuss some properties of Bessel functions.

\smallskip

For $\upsilon\in \mathbb{R},$ first kind Bessel function $J_\upsilon$ of order $\upsilon$ is defined by
\begin{align}\label{exp172}
J_\upsilon(t)=\sum_{l=0}^\infty \dfrac{(-1)^l (\frac{t}{2})^{\upsilon+2l}}{l!\, \Gamma(l+\upsilon+1)},
\end{align}
where $t$ is a non-negative real number. Then for $n\geq 2,$ $a\in (0,\infty)$ and $x\in\mathbb R^n,$
the following relation holds
\begin{align}\label{exp175}
\int_{S^{n-1}} e^{ia x\cdot \omega} d\omega=c_n \, (a|x|)^{1-\frac{n}{2}} J_{\frac{n}{2}-1}(a|x|),
\end{align}
where $c_n>0$ is a constant depending only on $n$ and $d\omega$ is the surface measure on $S^{n-1}.$

If $\xi,\zeta$ and $\eta$ are complex parameters with $\eta\neq 0,-1,\ldots,$ then the complex power series
\begin{align}\label{exp170}
\sum_{l=0}^\infty \dfrac{(\xi)_l(\zeta)_l}{(\eta)_l}\,\frac{z^l}{l\,!},
\end{align}
converges for $|z|<1,$ where $(\xi)_0=1$ and $(\xi)_l=\xi(\xi+1)\cdots(\xi+l-1)$ for $l\geq 1.$
The sum of the above series is denoted by ${}_2F_1(\xi,\zeta;\eta \, ;z)$ and called it hypergeometric function.
Moreover, if Re$(\eta-\xi-\zeta)>0,$ then the series (\ref{exp170}) also converges for $|z|=1$ and
\begin{align}\label{exp171}
{}_2F_1(\xi,\zeta;\eta \,;1)=\frac{\Gamma(\eta)\Gamma(\eta-\xi-\zeta)}{\Gamma(\eta-\xi) \Gamma(\eta-\zeta)}.
\end{align}
For more details, see \cite{AAR}, page no. 62-66.

The following property of Bessel functions hold, see \cite{W}, page no. 385.
\begin{lemma}\cite{W}\label{lemma170}
Let $u,\alpha,\upsilon\in\mathbb R$ be such that $u,\alpha>0$ and $\upsilon\geq 0.$ Then
\begin{align*}
\int_0^\infty e^{-ut}t^{\alpha-1}J_\upsilon(t)dt
&=\dfrac{(\frac{1}{2})^\upsilon \Gamma(\alpha+\upsilon)}{(1+u^2)^{\frac{\alpha+\upsilon}{2}}
\Gamma(\upsilon+1)}\,{}_2F_1\left(\frac{\alpha+\upsilon}{2},\frac{1-\alpha+\upsilon}{2};1+\upsilon \, ; \frac{1}{1+u^2}\right),
\end{align*}
where ${}_2F_1$ is the hypergeometric function as defined by the series (\ref{exp170}).
\end{lemma}

Next lemma is needful to prove the main result in this section.

\begin{lemma}\label{lemma171}
Let $n\in\mathbb N\setminus \{1\}$ and $\beta\in(\frac{n}{3},\frac{n}{2}).$ Then there exists $C>0$
and $a_0>0,$ both independent of $\beta,$ such that
\begin{align*}
\int_0^a t^{\frac{n}{2}-\beta} J_{\frac{n}{2}-1}(t)dt\geq C
\end{align*}
for all $a\geq a_0.$
\end{lemma}

\begin{proof}
Let $\frac{n}{3}<\beta<\frac{n}{2}$ and $\alpha=\frac{n}{2}-\beta+1,$ then $1<\alpha<\frac{n}{6}+1.$
Again for $0<u<1,$ we have $\frac{1}{2}<\frac{1}{1+u^2}<1.$ Now, we shall find an independent of $u$
and positive lower bound of the integral define in Lemma \ref{lemma170}, when $\alpha=\frac{n}{2}-\beta+1$
and $\upsilon=\frac{n}{2}-1.$ For that, consider the following two cases.

\smallskip

\noindent\textbf{For $n=2$ :} Since $1<\alpha<\frac{4}{3},$ except for the first term, all other terms
in the series expansion of ${}_2F_1\left(\frac{\alpha}{2},\frac{1-\alpha}{2};1;\frac{1}{1+u^2}\right)$
are negative. Further, as $1-\frac{\alpha}{2}-\frac{1-\alpha}{2}>0,$ we get
\begin{align}\label{exp174}
{}_2F_1\left(\frac{\alpha}{2},\frac{1-\alpha}{2};1;\frac{1}{1+u^2}\right)&=1+\sum_{l=1}^\infty
\dfrac{(\frac{\alpha}{2})_l(\frac{1-\alpha}{2})_l}{(1)_l\,l!}\frac{1}{(1+u^2)^l} \nonumber \\
&\geq1+\sum_{l=1}^\infty\dfrac{(\frac{\alpha}{2})_l(\frac{1-\alpha}{2})_l}{(1)_l\,l!}.
\end{align}
From Lemma \ref{lemma170}, using (\ref{exp171} - \ref{exp174}), we have
\begin{align}\label{exp176}
\int_0^\infty e^{-ut}t^{1-\beta} J_0(t)dt&=\int_0^\infty e^{-ut}t^{\alpha-1} J_0(t)dt \nonumber \\
&=\dfrac{\Gamma(\alpha)}{(1+u^2)^{\alpha/2}}\, {}_2F_1
\left(\frac{\alpha}{2},\frac{1-\alpha}{2};1;\frac{1}{1+u^2}\right) \nonumber \\
&\geq \frac{\Gamma(2)}{2}\frac{\Gamma(\frac{1}{2})}{\Gamma(1-\frac{\alpha}{2})\Gamma(\frac{\alpha+1}{2})}>0,
\end{align}
where the last inequality true as $1\leq (1+u^2)^{\alpha/2}\leq 2^{\alpha/2}\leq 2$ for $1<\alpha<\frac{4}{3}.$

\smallskip

\noindent\textbf{For $n\geq 3$ :} Since $\frac{n}{3}<\beta<\frac{n}{2},$ all the terms in the series
expansion of ${}_2F_1\left(\frac{n-\beta}{2},\frac{\beta-1}{2};\frac{n}{2};\frac{1}{1+u^2}\right)$
are positive. Hence ${}_2F_1\left(\frac{n-\beta}{2},\frac{\beta-1}{2};\frac{n}{2};\frac{1}{1+u^2}\right)\geq 1.$
Thus from Lemma \ref{lemma170}, we have
\begin{align}\label{exp177}
\int_0^\infty e^{-ut}t^{\frac{n}{2}-\beta} J_{\frac{n}{2}-1}(t)dt
&=\dfrac{(\frac{1}{2})^{\frac{n}{2}-1} \Gamma(n-\beta)}{(1+u^2)^{\frac{n-\beta}{2}}
\Gamma(\frac{n}{2})}\,{}_2F_1\left(\frac{n-\beta}{2},\frac{\beta-1}{2};\frac{n}{2};\frac{1}{1+u^2}\right) \nonumber \\
&\geq \dfrac{(\frac{1}{2})^{\frac{n}{2}-1} \Gamma(n-\beta)}{2^{\frac{n-\beta}{2}}\Gamma(\frac{n}{2})} >0.
\end{align}

\smallskip

Consider arbitrary $n\geq 2.$ Since for $x>0,$ $\Gamma(x)$ is continuous, the lower bounds in (\ref{exp176})
and (\ref{exp177}) are independent of $\alpha,\beta.$ Hence for $0<u<1$ and $\frac{n}{3}<\beta<\frac{n}{2},$
there exists $a_0,C>0,$ independent of $u,\beta,$ such that for all $a\geq a_0,$
\begin{align*}
\int_0^a e^{-ut}t^{\frac{n}{2}-\beta} J_{\frac{n}{2}-1}(t)dt \geq C.
\end{align*}
Therefore, $\lim_{u\rightarrow 0^+}\int_0^a e^{-ut}t^{\frac{n}{2}-\beta} J_{\frac{n}{2}-1}(t)dt \geq C$ for
all $a\geq a_0.$ Thus $$\int_0^a t^{\frac{n}{2}-\beta} J_{\frac{n}{2}-1}(t) dt \geq C$$ for all $a\geq a_0$
and $\frac{n}{3}<\beta<\frac{n}{2}.$
\end{proof}

Now we prove the main results of the section.
\begin{theorem}
Consider the square integrable and compactly supported function
$f_\beta(z,e^{i\varphi})=\chi_{B_1(0)}(z)\frac{1}{|z|^\beta},$ where
$\beta <1,$ on $M(2).$ Then  for $1<q<2,$ there exist
$\beta\in(\frac{2}{3},1)$ such that $\int_0^\infty \|\hat{f_\beta}(a)\|_{S_q}^q a da=\infty.$
\end{theorem}

\begin{proof}
Let $e_0(\theta)=1,$ then $e_0\in L^2(S^1).$ For $a>0,$
\begin{align*}
\langle \hat{f_\beta}(a)e_0,e_0 \rangle &=\frac{1}{2\pi}\int_{-\pi}^\pi\hat{f_\beta}(a)
e_0(\theta)\overline{e_0(\theta)}d\theta \\
&=\frac{1}{2\pi}\int_{-\pi}^\pi\int_{-\pi}^\pi\int_{\mathbb{C}}\chi_{B_1(0)}(z)\frac{1}{|z|^\beta}
e^{i\text{Re}(ae^{i\theta}z)}e_0(\theta-\varphi)dz d\varphi d\theta \\
&=\int_{-\pi}^\pi\int_{0}^1\frac{1}{r^\beta}\left(\int_{-\pi}^\pi
e^{iar\text{Re}(e^{i\theta}e^{i\psi})}d\theta\right)rdr d\psi.
\end{align*}
In view  of (\ref{exp175}), we can write
\begin{align}\label{exp173}
\langle\hat{f_\beta}(a)e_0,e_0 \rangle=2\pi c_2 \int_{0}^1\frac{1}{r^{\beta-1}}J_0(ar)dr
=2\pi c_2 \, a^{\beta-2}\int_{0}^a\frac{1}{r^{\beta-1}}J_0(r)dr.
\end{align}
Hence by Lemma \ref{lemma171}, if $\beta \in (\frac{2}{3},1),$ then there exists $C>0$ and $a_0>0$ such
that $\langle\hat{f_\beta}(a)e_0,e_0 \rangle\geq 2\pi c_2 \, C a^{\beta-2}$ for all $a\geq a_0.$ Thus for $1<q<2,$
\begin{align*}
\int_0^\infty \|\hat{f_\beta}(a)\|_{S_q}^q ada &\geq \int_0^\infty \|\hat{f_\beta}(a)\|_{S_\infty}^q ada
\geq \int_0^\infty |\langle\hat{f_\beta}(a)e_0,e_0 \rangle|^q ada \\
&\geq \int_{a_0}^\infty |\langle\hat{f_\beta}(a)e_0,e_0 \rangle|^q ada \geq \tilde{C}\int_{a_0}^\infty a^{(\beta-2)q} ada.
\end{align*}
The integral $\int_{a_0}^\infty a^{(\beta-2)q+1}\,da$ is finite if and only if $(\beta-2)q+1<-1,$
that is, $\beta<2-\frac{2}{q}.$ Since $q<2,$ we have $2-\frac{2}{q}<1.$ If we choose
$\beta\in (\,\max\{\frac{2}{3},2-\frac{2}{q}\},1),$ the corresponding function
$f_\beta$ will be the required function.
\end{proof}

\begin{theorem}
Let $n\geq 3.$ Consider the square integrable and compactly supported function
$f_\beta(x,k)=\chi_{B_1(0)}(x)\frac{1}{|x|^\beta},$ where $\beta <\frac{n}{2},$
on $M(n).$ Then  for $1<q<2,$ there exists $\beta\in(\frac{n}{3},\frac{n}{2})$ such
that $\int_0^\infty \sum_{\sigma\in\hat{M}}d_\sigma\|\hat f(a,\sigma)\|_{S_q}^q a^{n-1}da=\infty.$
\end{theorem}

\begin{proof}
Let $\sigma_o\in\hat{M},$ be the trivial representation. Consider the function $\varphi_0$ on $K$
defined by $\varphi_0(k)=1$ for all $k\in K.$ Since $d_{\sigma_o}=1,$ we have
$\varphi_0\in L^2(K,\mathbb C^{d_{\sigma_o}\times d_{\sigma_o}}).$ Now for $a>0,$
\begin{align*}
\langle \hat f_\beta(a,\sigma_o)\varphi_0,\varphi_0 \rangle
&=\int_K\int_K\int_{\mathbb R^n}f_\beta(x,k)e^{ia\left\langle s^{-1}\cdot e_1,\, x\right\rangle}
\varphi_0(sk)\overline{\varphi_0(s)}\,dxdkds \\
&=\int_K\int_{\mathbb R^n}\chi_{B_1(0)}(x)\frac{1}{|x|^\beta}e^{ia\left\langle s^{-1}\cdot e_1,\, x\right\rangle}\,dxds.
\end{align*}
Again $S^{n-1}=\{s^{-1}\cdot e_1:s\in K\}.$ Therefore, using (\ref{exp175}) we get
\begin{align*}
\langle \hat f_\beta(a,\sigma_o)\varphi_0,\varphi_0 \rangle
&=\int_{S^{n-1}}\int_{\mathbb R^n}\chi_{B_1(0)}(x)\frac{1}{|x|^\beta}e^{ia\left\langle \omega,\, x\right\rangle}\,dxd\omega \\
&=c_n \int_{\mathbb R^n}\chi_{B_1(0)}(x)\frac{1}{|x|^\beta} \, (a|x|)^{1-\frac{n}{2}} J_{\frac{n}{2}-1}(a|x|) dx \\
&=c_n\int_0^1\int_{S^{n-1}}\frac{1}{r^\beta} \, (ar)^{1-\frac{n}{2}} J_{\frac{n}{2}-1}(ar)r^{n-1} dr d\omega \\
&=c_n \, a^{\beta-n}\int_0^a r^{\frac{n}{2}-\beta} J_{\frac{n}{2}-1}(r) dr.
\end{align*}
Hence by Lemma \ref{lemma171}, if $\beta \in (\frac{n}{3},\frac{n}{2}),$ then there exists $C>0$ and $a_0>0$
such that $\langle \hat f_\beta(a,\sigma_o)\varphi_0,\varphi_0 \rangle\geq c_n \, C a^{\beta-n}$ for all
$a\geq a_0.$ Thus for $1<q<2,$
{\footnotesize
\begin{align*}
\int_0^\infty \sum_{\sigma\in\hat{M}}d_\sigma\|\hat f(a,\sigma)\|_{S_q}^q a^{n-1}da
&\geq \int_0^\infty \|\hat f(a,\sigma_o)\|_{S_q}^q a^{n-1}da
\geq \int_0^\infty \left|\langle \hat f_\beta(a,\sigma_o)\varphi_0,\varphi_0 \rangle\right|^qa^{n-1}da \\
&\geq \int_{a_0}^\infty \left|\langle \hat f_\beta(a,\sigma_o)\varphi_0,\varphi_0 \rangle\right|^qa^{n-1}da
\geq \tilde{C} \int_{a_0}^\infty a^{(\beta-n)q+n-1}da.
\end{align*}}
The integral $\int_{a_0}^\infty a^{(\beta-n)q+n-1}da$ is infinite if $(\beta-n)q+n-1\geq -1,$
that is, $\beta\geq n-\frac{n}{q}.$  Since $q<2,$ we have $n-\frac{n}{q}<\frac{n}{2}.$ If we
choose $\beta\in (\,\max\{\frac{n}{3},n-\frac{n}{q}\},\frac{n}{2}),$ the corresponding function
$f_\beta$ will be the required function.
\end{proof}

\section{Heisenberg motion group}
The Heisenberg group $\mathbb H^n=\mathbb C^n\times\mathbb R$ is a step two
nilpotent Lie group having center $\mathbb R$ that equipped with the group law
\[(z,t)\cdot(w,s)=\left(z+w,t+s+\frac{1}{2}\text{Im}(z\cdot\bar w)\right).\]

\smallskip

By the Stone-von Neumann theorem, the infinite dimensional irreducible unitary
representations of $\mathbb H^n$ can be parameterized by $\mathbb R^\ast=\mathbb
R\smallsetminus\{0\}.$ That is, each $\lambda\in\mathbb R^\ast$ defines a
Schr\"{o}dinger representation $\pi_\lambda$ of $\mathbb H^n$ via
\[\pi_\lambda(z,t)\varphi(\xi)=e^{i\lambda t}e^{i\lambda(x\cdot\xi+\frac{1}{2}x\cdot y)}\varphi(\xi+y),\]
where $z=x+iy$ and $\varphi\in L^2(\mathbb{R}^n).$

\smallskip

Having chosen sublaplacian $\mathcal L$ of the Heisenberg group $\mathbb H^n$ and its
geometry, there is a larger group of isometries that commute with $\mathcal L$, known as
Heisenberg motion group. The Heisenberg motion group $G$ is the semidirect product of
$\mathbb H^n$ with the unitary group $K=U(n).$ Since $K$ defines a group
of automorphisms on $\mathbb H^n,$ via $k\cdot(z,t)=(kz,t),$ the group law on
$G$ can be expressed as
\[(z,t,k_1)\cdot(w,s,k_2)=\left(z+k_1w, t+s-\frac{1}{2}\text{Im} (k_1w\cdot\bar z),k_1k_2\right).\]
Since a right $K$-invariant function on $G$ can be thought as a function on $\mathbb H^n,$
the Haar measure on $G$ is given by $dg=dzdtdk,$ where $dzdt$ and $dk$ are the normalized
Haar measure on $\mathbb H^n$ and $K$ respectively.

\smallskip

For $k\in K$ define another set of representations of the Heisenberg group
$\mathbb H^n$ by $\pi_{\lambda,k}(z,t)=\pi_\lambda(kz,t).$ Since $\pi_{\lambda,k}$
agrees with $\pi_\lambda$ on the center of $\mathbb H^n,$  it follows by
the Stone-Von Neumann theorem for the Schr\"{o}dinger representation that
$\pi_{\lambda,k}$ is equivalent to $\pi_\lambda.$ Hence there exists an
intertwining operator $\mu_\lambda(k)$ satisfying
\[\pi_\lambda(kz,t)=\mu_\lambda(k)\pi_\lambda(z,t)\mu_\lambda(k)^\ast.\]
Then $\mu_\lambda$ can be thought of as a unitary
representation of $K$ on $L^2(\mathbb R^n),$ called
metaplectic representation.
Let $(\sigma,\mathcal H_\sigma)$ be an irreducible unitary representation of $K$
and $\mathcal H_\sigma=\text{span}\{e_j^\sigma:1\leq j\leq d_\sigma\}.$ For
$k\in K,$ the matrix coefficients of the representation $\sigma\in\hat K$ are
given by \[\varphi_{ij}^\sigma(k)=\langle\sigma(k)e_j^\sigma, e_i^\sigma\rangle,\]
where $i, j=1,\ldots,  d_\sigma.$

\smallskip

Let $\phi_\alpha^\lambda(x)=|\lambda|^{\frac{n}{4}}\phi_\alpha(\sqrt{|\lambda|}x);~\alpha\in\mathbb Z_+^n,$
where $\phi_\alpha$'s are the Hermite functions on $\mathbb R^n.$ Since for each $\lambda\in\mathbb R^\ast,$
the set $\{\phi_\alpha^\lambda : \alpha\in\mathbb Z_+^n \}$ forms an orthonormal basis for $L^2(\mathbb R^n),$
letting $P_m^\lambda=\text{span}\{\phi_\alpha^\lambda:~ |\alpha|=m\},$ $\mu_\lambda$ becomes an irreducible
unitary representation of $K$ on $P_m^\lambda.$ Hence, the action of $\mu_\lambda$ can be realized on
$P_m^\lambda$ by
\begin{equation}\label{exp50}
\mu_\lambda(k)\phi_\gamma^\lambda=\sum_{|\alpha|=|\gamma|}\eta_{\gamma\alpha}^\lambda(k)\phi_\alpha^\lambda,
\end{equation}
where $\eta_{\alpha\gamma}^\lambda$'s are the matrix coefficients of $\mu_\lambda(k).$
Define a bilinear form $\phi_{\alpha}^\lambda\otimes e_j^\sigma$ on
$L^2(\mathbb R^n)\times\mathcal H_\sigma$ by
$\phi_{\alpha}^\lambda\otimes e_j^\sigma=\phi_{\alpha}^\lambda e_j^\sigma.$ Then
$\{\phi_{\alpha}^\lambda\otimes e_j^\sigma: \alpha\in\mathbb Z_+^n,1\leq j\leq d_\sigma\}$
forms an orthonormal basis for $L^2(\mathbb R^n)\otimes\mathcal{H}_\sigma.$
Denote  $\mathcal H_\sigma^2=L^2(\mathbb R^n)\otimes\mathcal{H}_\sigma.$

\smallskip

Define a representation $\rho_\sigma^\lambda$
of $G$ on the space $\mathcal H_\sigma^2$ by
\[\rho_\sigma^\lambda(z,t,k)=\pi_\lambda(z,t)\mu_\lambda(k)\otimes\sigma(k).\]
In the article \cite{S}, it is shown that $\rho_\sigma^\lambda$ are all possible
irreducible unitary representations of $G$ that participate in the Plancherel formula.
Thus, in view of the above discussion, we shall denote the partial dual of the group $G$ by
$ G'\cong\mathbb R^\ast\times\hat K.$ The Fourier transform of $f\in L^1(G)$ defined by
\[\hat f(\lambda,\sigma) =\int_K\int_\mathbb{R}\int_{\mathbb{C}^n} f(z,t,k)\rho_\sigma^\lambda(z,t,k)dzdtdk,\]
is a bounded linear operator on $\mathcal H_\sigma^2.$ As the Plancherel formula
\[\int_K\int_{\mathbb{H}^n} |f(z,t,k)|^2 dzdtdk=(2\pi)^{-n}\int_{\mathbb{R}\setminus\{0\}}
\sum_{\sigma\in\hat K}d_\sigma\|\hat f(\lambda,\sigma)\|^2_{S_2}|\lambda|^nd\lambda\]
holds for $f\in L^2(G),$ it follows that $\hat f(\lambda,\sigma)$ is a Hilbert-Schmidt
operator on $\mathcal H_\sigma^2.$ For detailed Fourier analysis on the Heisenberg motion group, see \cite{BJR, GS2, S}.

\smallskip

First, we recall the example of the previously mentioned required function for the Heisenberg group
and show that the example can be extended to the Heisenberg motion group.

Let $A=\{(x,y,t):|x_l|\leq 1,|y_l|\leq 1,|t|\leq 1; l=1,\ldots,n\}$ be a compact subset of
$\mathbb{R}^{2n}\times\mathbb{R}$ and
\begin{align}\label{exp180}
g_\xi(z,t)=g_\xi(x,y,t)=|t|^\xi\prod\limits_{j=1}^n|x_j|^\xi\prod\limits_{j=1}^n|y_j|^\xi\chi_{A}(z,t),
\end{align}
where $\xi>-\frac{1}{2}.$ Then the following result holds.
\begin{theorem}\cite{PZ}\label{th180}
For $1<q<2,$ there exists $\xi>-\frac{1}{2},$ such that
$$\int_{\mathbb{R}\setminus\{0\}}|\langle \hat{g_\xi}(\lambda)\phi_0^\lambda,\phi_0^\lambda \rangle|^q|\lambda|^n d\lambda$$
is infinite.
\end{theorem}

The following proposition gives the required function for the Heisenberg motion group.

\begin{theorem}
Consider the function $f_\xi(z,t,k)=g_\xi(z,t)$ on $G,$ where $g_\xi$ is defined in (\ref{exp180}). Then
$f_\xi$ is square integrable and compactly supported. Further for $1<q<2,$ there exists $\xi>-\frac{1}{2},$
such that $\int_{\mathbb R\setminus\{0\}}\sum_{\sigma\in\hat K}d_\sigma
\|\hat f_\xi(\lambda,\sigma)\|^q_{S_q}|\lambda|^n d\lambda$
is infinite
\end{theorem}

\begin{proof}
Let $1<q<2$ and $\lambda\in \mathbb{R}\setminus\{0\}.$ Then
\begin{align}\label{exp181}
\sum_{\sigma\in\hat K}d_\sigma \|\hat{f_\xi}(\lambda,\sigma)\|^q_{S_q}
&\geq \sum_{\sigma\in\hat K}\left(d_\sigma^{1/q} \|\hat{f_\xi}(\lambda,\sigma)\|_{S_2}\right)^q
\geq \left(\sum_{\sigma\in\hat K}\left(d_\sigma^{1/q} \|\hat{f_\xi}(\lambda,\sigma)\|_{S_2}\right)^2\right)^{q/2} \nonumber \\
&\geq \left(\sum_{\sigma\in\hat K}d_\sigma^{2/q} \|\hat{f_\xi}(\lambda,\sigma)\|_{S_2}^2\right)^{q/2}
\geq \left(\sum_{\sigma\in\hat K}d_\sigma \|\hat{f_\xi}(\lambda,\sigma)\|_{S_2}^2\right)^{q/2}.
\end{align}
It is already discussed that
$\{\phi_{\alpha}^\lambda\otimes e_j^\sigma: \alpha\in\mathbb Z_+^n,1\leq j\leq d_\sigma\}$
forms an orthonormal basis for $L^2(\mathbb R^n)\otimes\mathcal{H}_\sigma.$ Thus by (\ref{exp50}),
\[\langle\hat{f_\xi}(\lambda,\sigma)(\phi_{\alpha}^\lambda\otimes e_j^\sigma),(\phi_{\beta}^\lambda\otimes e_l^\sigma)\rangle
=\int_{\mathbb{H}^n}\int_K f_\xi(z,t,k)\sum\limits_{|\gamma|=|\alpha|}
\eta_{\alpha\gamma}(k)\Phi_{\gamma\beta}^\lambda(z,t)\phi_{lj}^\sigma (k)dzdtdk,\]
where $\Phi_{\gamma\beta}^\lambda(z,t)=\langle \pi_\lambda(z,t)\phi_\gamma^\lambda,\phi_\beta^\lambda\rangle.$
Therefore, $\|\hat{f_\xi}(\lambda,\sigma)(\phi_{\alpha}^\lambda\otimes e_j^\sigma)\|_2^2$ is equal to
\begin{align*}
&\sum_{\beta\in \mathbb Z_+^n}\sum_{1\leq l\leq d_\sigma}\left|\int_{\mathbb{H}^n}\int_K f_\xi(z,t,k)
\sum\limits_{|\gamma|=|\alpha|} \eta_{\alpha\gamma}(k)\Phi_{\gamma\beta}^\lambda(z,t)\phi_{lj}^\sigma(k)dzdtdk\right|^2 \\
=&\sum_{\beta\in \mathbb Z_+^n}\sum_{1\leq l\leq d_\sigma}\left|\sum\limits_{|\gamma|=|\alpha|}
\langle\hat{g_\xi}(\lambda)\phi_\gamma^\lambda,\phi_\beta^\lambda\rangle\int_K
\eta_{\alpha\gamma}(k)\phi_{lj}^\sigma(k)dk\right|^2.
\end{align*}
That is,
\begin{align*}
\|\hat{f_\xi}(\lambda,\sigma)\|_{S_2}^2=\sum_{\alpha,\beta\in \mathbb Z_+^n}\sum_{1\leq j,l\leq d_\sigma}
\left|\sum\limits_{|\gamma|=|\alpha|}\langle\hat{g_\xi}(\lambda)\phi_\gamma^\lambda,\phi_\beta^\lambda\rangle
\int_K\eta_{\alpha\gamma}(k)\phi_{lj}^\sigma(k)dk\right|^2.
\end{align*}
Hence by the Peter-Weyl theorem (Plancherel) for compact groups, we get
\begin{align*}
\sum_{\sigma\in\hat K}d_\sigma \|\hat{f_\xi}(\lambda,\sigma)\|_{S_2}^2
&=\sum_{\alpha,\beta\in \mathbb Z_+^n}\int_K\left|\sum\limits_{|\gamma|=|\alpha|}
\langle\hat{g_\xi}(\lambda)\phi_\gamma^\lambda,\phi_\beta^\lambda\rangle \eta_{\alpha\gamma}(k)\right|^2 dk \\
&\geq \sum_{\alpha,\beta\in \mathbb Z_+^n\setminus \mathbb{N}^n}\int_K\left|\sum\limits_{|\gamma|=|\alpha|}
\langle\hat{g_\xi}(\lambda)\phi_\gamma^\lambda,\phi_\beta^\lambda\rangle\eta_{\alpha\gamma}(k)\right|^2 dk \\
&=\int_K\left|\langle\hat{g_\xi}(\lambda)\phi_0^\lambda,\phi_0^\lambda\rangle\eta_{00}(k)\right|^2 dk
=|\langle\hat{g_\xi}(\lambda)\phi_0^\lambda,\phi_0^\lambda\rangle|^2.
\end{align*}
Since $\mu_\lambda|_{P_0^\lambda}$ is irreducible, the last equality follows from Schur's orthogonality relation.
Therefore, from (\ref{exp181}), we can conclude that
\begin{align*}
\int_{\mathbb R\setminus\{0\}}\sum_{\sigma\in\hat K}d_\sigma \|\hat{f_\xi}(\lambda,\sigma)\|^q_{S_q}|\lambda|^n d\lambda
&\geq\int_{\mathbb{R}\setminus\{0\}}\left(\sum_{\sigma\in\hat K}
d_\sigma\|\hat{f_\xi}(\lambda,\sigma)\|_{S_2}^2\right)^{q/2}|\lambda|^nd\lambda \\
&\geq\int_{\mathbb R \setminus\{0\}}|\langle\hat{g_\xi}(\lambda)\phi_0^\lambda,\phi_0^\lambda\rangle|^q |\lambda|^n d\lambda.
\end{align*}
Thus Theorem \ref{th180} proves the result.
\end{proof}

\begin{remark}
As compare to motion groups, the result in product space is straightforward. Let $G$ be a second countable
type $I$ locally compact unimodular group and $(\pi,\mathcal{H}_\pi)$ be its representation. Consider
$G_P=\mathbb{R}^n\times G.$ For $f\in L^1(G_P),$ Fourier transform defined by
$\hat{f}(x,\pi)=\int_{\mathbb{R}^n}\int_G f(y,u)e^{-2\pi ix\cdot y}\pi(u)d\nu(u)dy$
is a bounded operator on $\mathcal{H}_\pi.$ If we take $f(x,u)=f_1(x)f_2(u),$ then
\begin{align}\label{exp182}
\int_{\mathbb{R}^n}\int_{\hat{G}}\|\hat{f}(x,\pi)\|_{S_q}^q d\mu(\pi)dy
=\int_{\mathbb{R}^n}|\hat{f}_1(y)|^q dy \int_{\hat{G}}\|\hat{f}_2(\pi)\|_{S_q}^q d\mu(\pi).
\end{align}
It is known that, for $q\in (1,2),$ there exists a positive, square integrable and compactly supported function
$f_1$ on $\mathbb{R}^n$ such that $\int_{\mathbb{R}^n}|\hat{f}_1(y)|^q dy$ is infinite (see \cite{Si}).
If we choose a square integrable and compactly supported function $f_2$ on $G$ such that
$\int_{\hat{G}}\|\hat{f}_2(\pi)\|_{S_q}^q d\mu(\pi)\neq 0,$
then by (\ref{exp182}), $\int_{\mathbb{R}^n}\int_{\hat{G}}\|\hat{f}(x,\pi)\|_{S_q}^q d\mu(\pi)dy$
is infinite. Hence in view of Proposition \ref{prop155}, the Weyl transform $W_\zeta$ on
$G_P$ is not bounded for $\zeta\in L^p(G_P\times\widehat{G_P})$ with $2<p<\infty$ .
\end{remark}

\noindent{\bf Acknowledgements:}
The first author gratefully acknowledges the support provided by IIT Guwahati, Government of India.



\begin{thebibliography}{1000}

\bibitem{AAR} G. E. Andrews, R. Askey and R. Roy, {\em Special functions,}
 Encyclopedia of Mathematics and its Applications, Cambridge University Press, Cambridge, (71) 1999.

\bibitem{BJR} C. Benson, J. Jenkins and G. Ratcliff, {\em Bounded K-spherical functions on Heisenberg groups,}
J. Funct. Anal. 105 (1992), no. 2, 409-443.

\bibitem{CZ} L. Chen and J. Zhao, {\em Weyl transform and generalized spectrogram associated
with quaternion {H}eisenberg group,} Bull. Sci. Math. 136 (2012), no. 2, 127-143.

\bibitem{DM} M. Duflo and C. C. Moore, {\em On the regular representation of a nonunimodular locally compact group,}
J. Funct. Anal. 21 (1976), no. 2, 209-243.

\bibitem{GS2} S. Ghosh and R. K. Srivastava, {\em Benedicks-Amrein-Berthier theorem for the Heisenberg motion group
and quaternion Heisenberg group,} \href{https://arxiv.org/abs/1904.04023}{arXiv:1904.04023}.

\bibitem{KK} K. Kumahara and K. Okamoto, {\em An analogue of the Paley-Wiener theorem for the Euclidean motion group,}
Osaka Math. J. 10 (1973), 77-91.

\bibitem{N} W. Nasserddine, {\em $L^p$-Fourier inversion formula on certain locally compact groups,}
C. R. Math. Acad. Sci. Paris 357 (2019), no. 7, 583-588.

\bibitem{PZ} L. Peng and J. Zhao, {\em Weyl transforms associated with the Heisenberg group,}
Bull. Sci. Math. 132 (2008), no. 1, 78-86.

\bibitem{PZ2} L. Peng and J. Zhao, {\em Weyl transforms on the upper half plane,}
Rev. Mat. Complut. 23 (2010), no. 1, 77-95.

\bibitem{ST} R. P. Sarkar and S. Thangavelu, {\em On theorems of Beurling and Hardy for the Euclidean motion group,}
Tohoku Math. J. (2) 57 (2005), no. 3, 335-351.

\bibitem{S} S. Sen, {\em Segal-Bargmann transform and Paley-Wiener theorems on Heisenberg motion groups,}
Adv. Pure Appl. Math. 7 (2016), no. 1, 13-28.

\bibitem{Si} B. Simon, {\em The Weyl transforms and $L^p$ functions on phase space,}
Proc. Amer. Math. Soc. 116 (1992), 1045-1047.

\bibitem{Su} M. Sugiura, {\em Unitary representations and harmonic analysis}, North-Holland Mathematical Library,
North-Holland Publishing Co., Amsterdam; Kodansha, Ltd., Tokyo, (44) 1990.

\bibitem{W} G. N. Watson, {\em A treatise on the theory of Bessel functions,} second edition, Cambridge
University Press, Cambridge, 1944.

\bibitem{We} H. Weyl, {\em The Theory of Groups and Quantum Mechanics,} Dover, 1950.

\bibitem{Wo}M. W. Wong, {\em Weyl Transform,} Springer-Verlag, New York, 1998.

\end{thebibliography}
\end{document}